\documentclass[a4paper,11pt,reqno]{amsart}
\usepackage{amssymb}
\usepackage{amsmath,enumerate}
\usepackage{graphicx}
\usepackage{tikz}

\usepackage[top=30truemm,bottom=30truemm,left=20truemm,right=20truemm]{geometry}
\allowdisplaybreaks[4]

\newtheorem{Def}{Definition}[section]
\theoremstyle{remark}
\newtheorem{Rem}[Def]{Remark}

\newtheorem{Ass}[Def]{Assumption}
\newtheorem*{Ack}{Acknowledgments}
\theoremstyle{plain}
\newtheorem{Th}[Def]{Theorem}
\newtheorem{Prop}[Def]{Proposition}
\newtheorem{Lem}[Def]{Lemma}

\newcommand{\Z}{\mathbb{Z}}

\newcommand{\C}{\mathbb{C}}
\renewcommand{\P}{\mathbb{P}}
\renewcommand{\H}{\mathbb{H}}

\newcommand{\CH}{\mathcal{H}}
\newcommand{\CL}{\mathcal{L}}

\newcommand{\ga}{\gamma }
\newcommand{\de}{\delta }
\newcommand{\Ga}{\Gamma }
\newcommand{\De}{\Delta }
\newcommand{\vth}{\vartheta }

\newcommand{\vph}{\varphi }
\newcommand{\la}{\lambda }
\newcommand{\La}{\Lambda }
\newcommand{\om}{\omega }
\newcommand{\Om}{\Omega }

\newcommand{\na}{\nabla }
\newcommand{\pa}{\partial }
\newcommand{\ot}{\otimes }

\newcommand{\bu}{\bullet}
\newcommand{\we}{\wedge}

\newcommand{\DS}{\displaystyle }

\newcommand{\tpi}{2\pi \sqrt{-1}}
\newcommand{\pii}{\pi \sqrt{-1}}
\newcommand{\frs}{\mathfrak{s}}

\newcommand{\ch}{{\rm ch}}

\DeclareMathOperator{\Res}{Res}

\makeatletter
\@addtoreset{equation}{section}

\makeatother

\title[Riemann-Wirtinger integrals on the product of two tori]{
Riemann-Wirtinger integrals on the product of two one-dimensional complex tori  
}
\author[Y. Goto]{Yoshiaki Goto}
\address[Goto]{
  Otaru University of Commerce, 
  3-5-21, Midori, Otaru, Hokkaido, 047-8501, Japan
}
\email{goto@res.otaru-uc.ac.jp}

\date{\today}

\begin{document}
\begin{abstract}
  The Riemann-Wirtinger integral is an analogue of the hypergeometric integral 
  defined on a one-dimensional complex torus. 
  As a generalization, we define the Riemann-Wirtinger integral on 
  the product of two one-dimensional complex tori.
  We study the structure of the twisted cohomology group associated with the Riemann-Wirtinger integral 
  and derive a system of differential equations satisfied by this integral. 
\end{abstract}

\maketitle

\section{Introduction}
The Gauss hypergeometric function ${}_2 F_1 (a,b,c;z)$ admits an integral representation
\begin{align*}
  {}_2 F_1 (a,b,c;z)=\frac{\Ga(c)}{\Ga(a) \Ga(c-a)} \int_0^1 u^a (1-u)^{c-a} (1-zu)^{-b} \frac{du}{u(1-u)} .
\end{align*}
This integral can be interpreted as a pairing of a twisted homology class 
and a twisted cohomology class on $\P^1_u -\{ 0,1,1/z,\infty \}$. 
Aomoto (e.g., \cite{AK}) generalized this framework to 
twisted (co)homology theory on $\P^n$ minus divisors, which enables the systematic study of 
many types of hypergeometric functions. 

These theories provide generalizations from $\P^1$ to $\P^n$. 
On the other hand, one may consider another type of generalization, namely from $\P^1$ to a Riemann surface 
with positive genus. 
The Riemann-Wirtinger integral gives such a generalization on a one-dimensional complex torus. 
The Riemann-Wirtinger integral in \cite{Mano}, \cite{Mano-Watanabe} is defined by 
\begin{align*}
  \int_{\ga} e^{\tpi c_0 u}\vth_1(u-t_1)^{c_1} \cdots \vth_1(u-t_n)^{c_n} \frs(u-t_j;\la)du ,
\end{align*}
where $\ga$ is a twisted cycle, $c_0\in \C$, $c_1,\dots ,c_n,\la \in \C-\Z$, 
and the theta function $\vth_1$ and $\frs(u;\la)$ are defined in \S \ref{sec:preliminaries}.
The integrand can be regarded as a multivalued function on 
$M$, a complex torus with $n$ points removed. 
It defines a local system $\CL_{\la}$ on $M$. 
Thus, the Riemann-Wirtinger integral can be studied in terms of 
twisted (co)homology groups. 
In \cite{Mano-Watanabe}, 
the structures of the twisted (co)homology groups are investigated in detail, 
and in \cite{G-RWintegral-intersection}, 
the intersection theory on these groups are studied. 

In this paper, 
we consider a generalization of the Riemann-Wirtinger integral to the two-dimensional case. 
As a first example, we consider a local system on the product of two one-dimensional complex tori. 
In particular, we choose a simple example of such local systems, 
and analyze it in detail. 
Using the multivalued function $T(u_1,u_2)$ defined in (\ref{eq:T-def}), 
we define a generalization of the Riemann-Wirtinger integral. 
We study the structure of the twisted cohomology group associated with this integral. 
First, we construct a basis of the twisted cohomology group using differential $2$-forms 
introduced in \cite{LV} (Section \ref{sec:cohomology}). 
Next, we derive a system of differential equations satisfied by the integrals 
(Section \ref{sec:diff-eq}). 
Although we do not study properties of this system or the structure of the twisted homology groups, 
we leave these problems for future work. 


\section{Preliminaries}\label{sec:preliminaries}
In this section, we introduce several functions that we will use throughout this paper. 

We define the theta function 
\begin{align*}
  \vth_1 (u)=\vth_1 (u,\tau)
  =-\sum_{m\in \Z} \exp\left(\pi \sqrt{-1} \Big( m+\frac{1}{2} \Big)^2 \tau
  +\tpi \Big( m+\frac{1}{2} \Big)\Big( u+\frac{1}{2} \Big)\right), 
\end{align*}
where $u\in \C$ and $\tau \in \H$. 
Throughout this paper, we fix $\tau \in \H$ and use the notation $\vth_1 (u)$. 
It is well known that $\vth_1 (u)$ is an odd function, has a simple zero at $u=0$, and 
satisfies the quasi-periodicity 
\begin{align*}
  \vth_1(u+1)=-\vth_1(u) =e^{\pii}\vth_1(u),\quad 
  \vth_1(u+\tau)=-e^{-\pii (\tau+2u)}\vth_1(u)=e^{-\pii (\tau+2u+1)}\vth_1(u).
\end{align*}
We also introduce the following two functions: 
\begin{align*}
  \rho (u)=\frac{\vth_1'(u)}{\vth_1 (u)} ,\qquad 
  \frs (u;\la) =\frac{\vth_1 (u-\la) \vth_1'(0)}{\vth_1(u)\vth_1(-\la)} ,
\end{align*}
where $\vth_1'(u)=\frac{d}{du}\vth_1(u)$ and $\la \in \C-\Z$. 
Note that $\rho (u)$ is an odd function, and
$\frs (u;\la)$ satisfies the quasi-periodicity
\begin{align*}
  \frs (u+1;\la) =\frs (u;\la) ,\quad 
  \frs (u+\tau;\la) =e^{\tpi \la} \frs (u;\la) .
\end{align*}

\section{Twisted cohomology group}\label{sec:cohomology}
\subsection{Settings}
For a fixed $\tau \in \H$, we set $\La_{\tau}=\Z +\Z \tau$ and $E=\C/\La_{\tau}$. 
We consider the product $E^2=E\times E$. 
For $\la=(\la_1,\la_2) \in \C^2$, we define a one-dimensional representation 
$e_{\la} :\pi_1 (E^2) \simeq \La_{\tau}^2 \to \C^{*}$ of 
the fundamental group $\pi_1 (E^2)$
by  
\begin{align*}
  e_{\la}(l_1+m_1\tau,l_2+m_2\tau) =e^{\tpi (m_1\la_1+m_2\la_2)} \qquad 
  (l_1,m_1,l_2,m_2 \in \Z). 
\end{align*}
Let $R_{\la}$ be the local system on $E^2$ determined by this representation $e_{\la}$.  

Let $n_1,n_2 \geq 1$ and assume that 
$t_{11},\dots ,t_{1n_1},t_{21},\dots ,t_{2n_2} \in \C$ represent distinct points of $E$. 
We also assume that $t_{1i}+t_{2j} \not\in \La_{\tau}$ for all $i,j$. 
We consider the following hyperplanes in $\C^2$: 
\begin{align*}
  \tilde{H}_{\pm} = (u_1 \pm u_2 =0) ,\quad 
  \tilde{H}_{kj} =(u_k =t_{kj}) \ (k\in \{ 1,2\} ,\ j\in \{ 1,\dots ,n_k \}).  
\end{align*}
These hyperplanes define \textit{elliptic hyperplanes}\footnote{
  The name is introduced in \cite{LV}. 
} $H_{\pm}$, $H_{kj}$ in $E^2$. 
For example, $H_{11}$ is defined as 
$\{ ([u_1],[u_2]) \in E^2 \mid [u_1]=[t_{11}] \}$, 
where $[t]$ denotes the point of $E$ represented by $t \in \C$. 
We set $\CH =\{ H_{\pm} \} \cup \{ H_{kj} \}_{k,j}$ and 
$M=E^2 -\bigcup_{H\in \CH} H$. 

Levin and Varchenko \cite{LV} studied the cohomology group $H^l (X;R_{\la})$, 
where $X$ is the complement of elliptic hyperplanes in $E^k$, 
and $R_{\la}$ is the local system defined similarly to the above. 
Our situation is a special case of their framework. 
In this paper, we consider a multivalued function on $M$ and 
the local system defined by this function. 

We define the multivalued function 
\begin{align}
  \nonumber
  T(u_1,u_2)=
  &e^{\tpi (c_{10}u_1 +c_{20}u_2)} \cdot \vth_1(u_1-u_2)^{c} \cdot \vth_1(u_1+u_2)^{c} \\
  \label{eq:T-def}
  &\cdot \vth_1(u_1-t_{11})^{c_{11}} \cdots \vth_1(u_1-t_{1n_1})^{c_{1n_1}} 
  \cdot \vth_1(u_2-t_{21})^{c_{21}} \cdots \vth_1(u_2-t_{2n_2})^{c_{2n_2}} 
\end{align}
on $M$, where 
$c_{10},c_{20}\in \C$, and $c,c_{11},\dots ,c_{1n_1},c_{21},\dots ,c_{2n_2}\in \C-\Z$. 
We set 
\begin{align*}
  c_{k\infty} =-\la_k -c_{k0}\tau -\sum_{j=1}^{n_k} c_{kj} t_{kj} \quad 
  (k \in \{ 1,2 \}).
\end{align*}
Throughout this paper, we impose the following conditions. 
\begin{Ass}
  \begin{align}
    \label{eq:ass-lambda}
    &\la_1,\ \la_2 ,\ \la_1 +\la_2 ,\ \la_1-\la_2 \not\in \La_{\tau} ;\\
    \label{eq:ass-sum-c}
    &2c+c_{11}+\cdots +c_{1n_1}=2c+c_{21}+\cdots +c_{2n_2}=0 ; \\
    \label{eq:ass-c_inf}
    &\text{$c_{1\infty}$ and $c_{2\infty}$ are constant.}
  \end{align}
\end{Ass}
By the condition (\ref{eq:ass-lambda}), $\la$ is convenient for $\CH$ 
in the sense of \cite{LV}.
By the conditions (\ref{eq:ass-sum-c}) and (\ref{eq:ass-c_inf}), we have 
\begin{align}
  \nonumber
  &T(u_1+1,u_2)=e^{\tpi c_{10}}T(u_1,u_2),&  
  &T(u_1+\tau,u_2)=e^{-\tpi (c_{1\infty}+\la_1)}T(u_1,u_2),& \\ 
  \label{eq:T-tau-shift}
  &T(u_1,u_2+1)=e^{\tpi (c_{20} -c)}T(u_1,u_2), &
  &T(u_1,u_2+\tau)=e^{-\tpi (c_{2\infty}+\la_2 -c)}T(u_1,u_2),&  
\end{align}
and thus $T(u)$ defines local systems $\CL=\C T(u)^{-1}$ and $\CL^{\vee}=\C T(u)$ on $M$. 

\begin{Rem}
  We regard our $T$ as the simplest example of a two-dimensional version of the Riemann-Wirtinger integral.  
  Although it seems that the multivalued function 
  \begin{align*}
    &e^{\tpi (c_{10}u_1 +c_{20}u_2)} \cdot \vth_1(u_1-u_2)^{c}  \\
    &\cdot \vth_1(u_1-t_{11})^{c_{11}} \cdots \vth_1(u_1-t_{1n_1})^{c_{1n_1}} 
      \cdot \vth_1(u_2-t_{21})^{c_{21}} \cdots \vth_1(u_2-t_{2n_2})^{c_{2n_2}} 
  \end{align*}
  gives a simpler example, 
  this function does not define a local system on $M$, 
  since the exponential factors in (\ref{eq:T-tau-shift}) depend on $u_1$ and $u_2$. 
\end{Rem}

\subsection{Twisted cohomology group}
We set $\CL_{\la}=\CL \ot_{\C} R_{\la}$, and 
we consider the twisted cohomology group $H^i(M;\CL_{\la})$. 
We set $\Om^k_{\la}=\Om^k_{M} \ot_{\C} R_{\la}$ and 
we define a connection $\na :\Om^k_{\la} \to \Om^{k+1}_{\la}$ by $\na \vph =d\vph +\om \we \vph$, 
where $\om =d\log T(u)\in \Om^1_M(M)$. 
By \cite[Corollaire II.6.3]{Deligne}, we have 
$H^i(M;\CL_{\la})\simeq H^i (M,(\Om^{\bu}_{\la}(M),\na))$. 
By \cite{Cho}, this cohomology vanishes except in degree two. 
In what follows, we thus investigate the structure of 
$H^2(M;\CL_{\la})\simeq \Om^{2}_{\la}(M)/\na (\Om^{1}_{\la}(M))$, 
and we identify the two sides via this isomorphism. 
It is well known that the dimension of $H^2(M;\CL_{\la})$ coincides with 
the Euler characteristic $\chi(M)$ of $M$. 
We evaluate $\chi (M)$. 

\begin{Prop}\label{prop:dim-H^2}
  \begin{align*}
    \dim H^2(M;\CL_{\la}) =\chi(M)=n_1 n_2+2n_1+2n_2+4 =(n_1+2)(n_2+2).
  \end{align*}
\end{Prop}
\begin{proof}
  Since $\chi (E^2)=\chi (E)=0$,  
  the Euler characteristic $\chi(M)$ coincides with $\sum_{H\neq H'\in \CH} \# (H\cap H')$. 
  For $H\neq H' \in \CH$, we have $\# (H\cap H')=1$ except for $\{ H,H' \}=\{ H_{+},H_{-} \}$. 
  On the other hand, we have 
  $H_{+}\cap H_{-}=\{ P_m=(w_m,w_m) \}_{m=1}^{4}$, 
  where we set $w_1 =0$, $w_2 =\frac{1}{2}$, $w_3 =\frac{\tau}{2}$, $w_4 =\frac{1+\tau}{2}$. 
  We thus obtain $\# (H_{+}\cap H_{-})=4$, and the proof is completed. 
\end{proof}

\subsection{Differential forms}\label{subsec:differential-forms}
There are $\chi(M)=n_1 n_2+2n_1+2n_2+4$ points at which two elliptic hyperplanes intersect. 
For each intersection point, we construct an $R_{\la}$-valued differential $2$-form
that has logarithmic poles along $\CH$ and has a nonzero (resp. zero) iterated residue 
at this point (resp. the other intersection points). 

For $i\in \{ 1,\dots ,n_1 \}$, $j\in \{ 1,\dots ,n_2 \}$, we set
\begin{align*}
  \psi_{ij}
  &=\frs(u_1 -t_{1i};\la_1) du_1 \we \frs(u_2-t_{2j};\la_2) du_2 ,\\
  \psi_{i\pm}
  &=\frs(u_1-t_{1i};\la_1 \mp \la_2) du_1 \we \frs(u_1 \pm u_2; \pm \la_2) (du_1 \pm du_2) \\
  &=\pm \frs(u_1-t_{1i};\la_1 \mp \la_2) \frs(u_1 \pm u_2; \pm \la_2) du_1 \we du_2 , \\
  \psi_{\pm j}
  &=\frs(u_1 \pm u_2;\la_1) (du_1 \pm du_2) \we \frs(u_2-t_{2j};\la_2 \mp \la_1) du_2 \\
  &=\frs(u_1 \pm u_2;\la_1) \frs(u_2-t_{2j};\la_2 \mp \la_1) du_1 \we du_2 .
\end{align*}
It is easy to see that 
they have logarithmic poles along $\CH$, and 
$\psi_{ij}$ (resp. $\psi_{i\pm}$, $\psi_{\pm j}$) has iterated residue $1$ at 
$H_{1i}\cap H_{2j}$ (resp. $H_{1i} \cap H_{\pm}$, $H_{\pm}\cap H_{2j}$). 
For example, we have 
\begin{align*}
  \Res_{H_{2j}}\left( \Res_{H_{1i}} \psi_{ij}\right) 
  =\Res_{H_{2j}}\left( \frs(u_2-t_{2j};\la_2) du_2 \right)
  =1.
\end{align*}

We require four differential forms corresponding to 
$P_1,P_2,P_3,P_4 \in H_{+}\cap H_{-}$ (see the proof of Proposition \ref{prop:dim-H^2}). 
Following \cite[4.3]{LV}, we construct such differential forms. 
As in \cite[(4.7)]{LV}, we define the following four differential forms: 
\begin{align*}
  \psi'_{+-,1}
  &=\frs\Big(u_1+u_2;\frac{\la_1 +\la_2}{2}\Big) (du_1 +du_2) 
    \we \frs\Big(u_1-u_2;\frac{\la_1-\la_2}{2}\Big) (du_1 -du_2) \\
  &= -2\frs\Big(u_1+u_2;\frac{\la_1 +\la_2}{2}\Big) 
    \frs\Big(u_1-u_2;\frac{\la_1-\la_2}{2}\Big)du_1 \we  du_2 ,\\ 
  \psi'_{+-,2}
  &=\frs\Big(u_1+u_2;\frac{\la_1 +\la_2+1}{2}\Big) (du_1 +du_2) 
    \we \frs\Big(u_1-u_2;\frac{\la_1-\la_2+1}{2}\Big) (du_1 -du_2) \\
  &=-2 \frs\Big(u_1+u_2;\frac{\la_1 +\la_2+1}{2}\Big) 
    \frs\Big(u_1-u_2;\frac{\la_1-\la_2+1}{2}\Big)du_1 \we  du_2 ,\\
  \psi'_{+-,3}
  &=e^{-\tpi u_1}
    \frs\Big(u_1+u_2;\frac{\la_1 +\la_2+\tau}{2}\Big) (du_1 +du_2) 
    \we \frs\Big(u_1-u_2;\frac{\la_1-\la_2+\tau}{2}\Big) (du_1 -du_2) \\
  &=-2 e^{-\tpi u_1} \cdot 
    \frs\Big(u_1+u_2;\frac{\la_1 +\la_2+\tau}{2}\Big) 
    \frs\Big(u_1-u_2;\frac{\la_1-\la_2+\tau}{2}\Big)du_1 \we  du_2 ,\\
  \psi'_{+-,4}
  &=e^{-\tpi u_1}
    \frs\Big(u_1+u_2;\frac{\la_1 +\la_2+1+\tau}{2}\Big) (du_1 +du_2) 
    \we \frs\Big(u_1-u_2;\frac{\la_1-\la_2+1+\tau}{2}\Big) (du_1 -du_2) \\
  &=-2 e^{-\tpi u_1} \cdot 
    \frs\Big(u_1+u_2;\frac{\la_1 +\la_2+1+\tau}{2}\Big) 
    \frs\Big(u_1-u_2;\frac{\la_1-\la_2+1+\tau}{2}\Big)du_1 \we  du_2 . 
\end{align*}
Let $M$ be the matrix whose $(i,j)$-entry is the iterated residue of $\psi'_{+-,j}$ at $P_i$. 
Since we have 
\begin{align*}
  M=
  \begin{pmatrix}
    1&1&1&1 \\
    1&1&-1&-1 \\ 
    \ell&-\ell&\ell&-\ell \\ 
    \ell&-\ell&-\ell&\ell 
  \end{pmatrix}
  \ (\ell =e^{\pii (\la_1 +\la_2)}),\quad 
  M^{-1}=\frac{1}{4}
  \begin{pmatrix}
    1&1&\ell^{-1}&\ell^{-1} \\
    1&1&-\ell^{-1}&-\ell^{-1} \\ 
    1&-1&\ell^{-1}&-\ell^{-1} \\ 
    1&-1&-\ell^{-1}&\ell^{-1} 
  \end{pmatrix}, 
\end{align*}
we set 
$(\psi_{+-,1},\psi_{+-,2},\psi_{+-,3},\psi_{+-,4})
=(\psi'_{+-,1}, \psi'_{+-,2}, \psi'_{+-,3}, \psi'_{+-,4}) \cdot M^{-1}$, 
namely, 
\begin{align*}
  &\psi_{+-,1} =\frac{1}{4}(\psi'_{+-,1} + \psi'_{+-,2} + \psi'_{+-,3} + \psi'_{+-,4}),& 
  &\psi_{+-,2} =\frac{1}{4}(\psi'_{+-,1} + \psi'_{+-,2} - \psi'_{+-,3} - \psi'_{+-,4}),& \\
  &\psi_{+-,3} =\frac{\ell^{-1}}{4}(\psi'_{+-,1} - \psi'_{+-,2} + \psi'_{+-,3} -\psi'_{+-,4}),& 
  &\psi_{+-,4} =\frac{\ell^{-1}}{4}(\psi'_{+-,1} - \psi'_{+-,2} - \psi'_{+-,3} +\psi'_{+-,4} ).&
\end{align*}
By construction, 
$\psi_{+-,m}$ has iterated residue $\de_{im}$ at the intersection point $P_i$. 

Let $\langle \cdot ,\cdot \rangle_{\ch}$ denote the cohomology intersection form; 
see e.g., \cite{Matsubara-localization}. 
For a differential $2$-form $\vph (u;\la)$, we set $\vph (u;\la)^{\vee}=\vph (u;-\la)$, 
which defines an element of $H^2(M;\CL_{\la}^{\vee})\simeq \Om^{2}_{-\la}(M)/\na^{\vee} (\Om^{1}_{-\la}(M))$, 
where $\na^{\vee}=d-\om \we$. 
Let 
$\Psi =\{ (i,j) \mid i\in \{ 1,\dots ,n_1,+,-\} ,\ j\in \{ 1,\dots ,n_2,+,-\}\} \cup \{ (+-,m) \}_{m=1}^{4}$ 
be the index set. 
Note that the cardinality of $\Psi$ is $(n_1+2)(n_2+2)$. 
\begin{Th}
  \begin{enumerate}[(1)]
  \item For $\ast, \dagger\in \Psi$, $\ast \neq \dagger$, 
    we have $\langle \psi_{\ast} ,\psi_{\dagger}^{\vee} \rangle_{\ch}=0$. 
  \item For $i\in \{ 1,\dots ,n_1\}$, $j\in \{ 1,\dots ,n_2\}$ and $m\in \{1,\dots ,4\}$, we have 
    \begin{align*}
      &\langle \psi_{ij} ,\psi_{ij}^{\vee} \rangle_{\ch}=\frac{(\tpi)^2}{c_{1i}c_{2j}},\quad  
      \langle \psi_{i \pm} ,\psi_{i \pm}^{\vee} \rangle_{\ch}=\frac{(\tpi)^2}{c_{1i}c},\\ 
      &\langle \psi_{\pm j} ,\psi_{\pm j}^{\vee} \rangle_{\ch}=\frac{(\tpi)^2}{c_{2j}c},\quad  
      \langle \psi_{+-,m} ,\psi_{+-,m}^{\vee} \rangle_{\ch}=\frac{(\tpi)^2}{c^2}.
    \end{align*}
  \item The differential forms $\{ \psi_{*} \}_{\ast \in \Psi}$ form a basis of the twisted cohomology group
    $H^2(M;\CL_{\la})$. 
  \end{enumerate}
\end{Th}
\begin{proof}
  (1) and (2) follow from \cite{Matsubara-localization}. 
  Since the intersection matrix is diagonal with nonzero diagonal entries, 
  the set $\{ \psi_{*} \}_{\ast \in \Psi}$ is linearly independent. 
\end{proof}

\section{Differential equations}\label{sec:diff-eq}
We derive a system of differential equations with respect to $t$ satisfied by the integrals 
\begin{align*}
  F_{\ast}(t;\la) =\int_{\De} T(u) \psi_{\ast} \quad (\ast \in \Psi), 
\end{align*}
where $\De$ is a twisted cycle. 
We call these integrals the \emph{Riemann-Wirtinger integrals} on $M$. 

\begin{Rem}
  In this paper, we do not discuss twisted cycles in detail. 
  Although it is an interesting problem to construct a basis of the twisted homology group, 
  the author does not have a complete answer. 
  As examples, we give some twisted cycles. 
  For $a,b\in \{ 0,1,2,3 \}$, we set 
  $E_{ab}=\{ x+y\tau \in E \mid \frac{a}{4}<x<\frac{a+1}{4},\ \frac{b}{4}<y<\frac{b+1}{4}\}$. 
  We assume $\{ t_{1i} \}_{i=1}^{n_1} \subset E_{00}$ and $\{ t_{2j} \}_{j=1}^{n_2} \subset E_{11}$. 
  In $E - \{ t_{kj} \}_{j=1}^{n_k}$, let $\ga_{1j}^{(k)}$ ($j=2,\dots ,n_k$), $\ga_{10}^{(k)}$, $\ga_{1\infty}^{(k)}$ 
  be paths corresponding to $(t_{k1}, t_{kj})$, $(t_{k1},t_{k1}+1)$, $(t_{k1},t_{k1}+\tau)$, respectively 
  (see \cite{Mano}, \cite{G-RWintegral-intersection}). 
  Then, for $(i,j)\neq (0,\infty), (\infty,0)$, the products $\ga_{1i}^{(1)}\times \ga_{1j}^{(2)}$ define twisted cycles 
  with respect to $T(u)$. 
  However, these cycles are not sufficient to form a basis of the twisted homology group. 
\end{Rem}

\subsection{Main theorem}
\begin{Th}
  We set $a_{+}=e^{\tpi \la_2}$, $a_{-}=e^{\tpi (\la_1+\la_2)}$, 
  $b_{+}=e^{\tpi \la_1}$, $b_{-}=e^{\tpi (\la_1+\la_2)}(=a_{-})$, 
  and $\pa_{kj}=\pa / \pa t_{kj}$. 
  The functions $F_{\ast}(t;\la)$ ($\ast \in \Psi$) satisfy the following differential equations. 
  For $p\in \{ 1,\dots ,n_1 \}$, $i\in \{ 1,\dots ,n_1 \}-\{ p \}$, 
  $j\in \{ 1,\dots ,n_2\}$ and $m\in \{ 1,2,3,4 \}$, we have 
  \begin{align*}
    \pa_{1p} (F_{ij})
    &=c_{1p} \rho (t_{1p}-t_{1i}) F_{ij} -c_{1p}\frs(t_{1p}-t_{1i};\la_1) F_{pj} ,\\
    \pa_{1p} (F_{i\pm})
    &=c_{1p} \rho (t_{1p}-t_{1i}) F_{i\pm} -c_{1p} \frs(t_{1p}-t_{1i};\la_1 \mp \la_2) F_{p \pm} , \\
    \pa_{1p} (F_{\pm j})
    &=c_{1p} \rho (t_{1p} \pm t_{2j})F_{\pm j} 
      \mp c_{1p}\frs(t_{2j} \pm t_{1p};\pm \la_1)F_{pj}
      \mp c_{1p}\frs(\mp t_{1p}-t_{2j};\la_2 \mp \la_1)F_{p \pm} ,\\
    \pa_{1p} (F_{+-,m})
    &=c_{1p} \frs(t_{1p}-w_m;\la_1-\la_2)F_{p+} -c_{1p}\frs(t_{1p}-w_m;\la_1+\la_2)F_{p-} \\
    &\qquad +
    \begin{cases}
      c_{1p}\rho(t_{1p}-w_m) F_{+-,m} & (m=1,2) \\
      c_{1p}\big( \rho(t_{1p}-w_m) -\pii \big)F_{+-,m} & (m=3,4)
    \end{cases}
    ,\\
    \pa_{1p} (F_{pj} )
    &= 
      \big( \tpi c_{10} +\sum_{i\neq p} c_{1i}\rho (t_{1p}-t_{1i}) +c (\rho(t_{1p}-t_{2j}) +\rho(t_{1p}+t_{2j})) \big) F_{pj}
      +\sum_{i\neq p} c_{1i} \frs(t_{1i}-t_{1p};\la_1) F_{ij} \\
    &\qquad +c\frs(- t_{2j}-t_{1p};\la_1) F_{+j} 
      +c\frs( t_{2j}-t_{1p};\la_1) F_{-j} 
      +c\frs(-t_{1p} -t_{2j};\la_2) F_{p+} 
      -c\frs(t_{1p} -t_{2j};\la_2) F_{p-} ,\\
    \pa_{1p} (F_{p\pm} )
    &= 
      \Big( \tpi (c_{10} \mp c_{20}) - \sum_{i\neq p} c_{1i}\rho(t_{1i}-t_{1p}) 
      +\sum_{j} c_{2j}\rho (t_{1p} \pm t_{2j})      
      + 2c \rho(2t_{1p}) \Big) F_{p\pm} \\
    &\qquad -2c\frs(2t_{1p}; \pm \la_2) F_{p\mp}
      + \sum_{i\neq p} c_{1i} \frs(t_{1i}-t_{1p};\la_1 \mp \la_2) F_{i\pm} \\
    &\qquad -\sum_{j} \big( c_{2j}\frs (t_{1p} \pm t_{2j}; \pm \la_2) F_{pj} 
      +c_{2j} \frs(\mp t_{2j}-t_{1p};\la_1 \mp \la_2) F_{\pm j} \big) \\
    &\qquad \mp c \frs(-t_{1p};\la_1 \mp \la_2) F_{+-,1}
      \mp c \frs(\pm \frac{1}{2}-t_{1p};\la_1 \mp \la_2) F_{+-,2} \\
    &\qquad \mp a_{\pm} c \frs(\pm\frac{\tau}{2}-t_{1p};\la_1 \mp \la_2) F_{+-,3}
      \mp a_{\pm} c \frs(\pm \frac{1+\tau}{2}-t_{1p};\la_1 \mp \la_2) F_{+-,4} .
  \end{align*}
  For $q\in \{ 1,\dots ,n_2 \}$, $j\in \{ 1,\dots ,n_2 \}-\{ q \}$, 
  $i\in \{ 1,\dots ,n_1\}$ and $m\in \{ 1,2,3,4 \}$, we have 
  \begin{align*}
    \pa_{2q} (F_{ij})
    &=c_{2q} \rho (t_{2q}-t_{2j}) F_{ij} -c_{2q}\frs(t_{2q}-t_{2j};\la_2) F_{iq} ,\\
    \pa_{2q} (F_{\pm j})
    &=c_{2q} \rho (t_{2q}-t_{2j}) F_{\pm j} -c_{2q} \frs(t_{2q}-t_{2j};\la_2 \mp \la_1) F_{\pm q} , \\
    \pa_{2q} (F_{i\pm})
    &=c_{2q} \rho (t_{2q} \pm t_{1i})F_{i\pm} 
      \mp c_{2q}\frs(t_{1i} \pm t_{2q};\pm \la_2)F_{iq}
      \mp c_{2q}\frs(\mp t_{2q}-t_{1i};\la_1 \mp \la_2)F_{\pm q} ,\\
    \pa_{2q} (F_{+-,m})
    &=-c_{2q} \frs(t_{2q}-w_m;\la_2-\la_1)F_{+q} +c_{2q}\frs(t_{2q}-w_m;\la_2+\la_1)F_{-q} \\
    &\qquad +
    \begin{cases}
      c_{2q}\rho(t_{2q}-w_m) F_{+-,m} & (m=1,2) \\
      c_{2q}\big( \rho(t_{2q}-w_m) -\pii \big)F_{+-,m} & (m=3,4)
    \end{cases}
     ,\\
    \pa_{2q} (F_{iq} )
    &=
      \big( \tpi c_{20} +\sum_{j\neq q} c_{2j}\rho (t_{2q}-t_{2j}) +c(\rho(t_{2q}-t_{1i}) +\rho(t_{2q}+t_{1i})) \big) F_{iq}
      +\sum_{j\neq q} c_{2j} \frs(t_{2j}-t_{2q};\la_2) F_{ij} \\
    &\qquad +c\frs(- t_{1i}-t_{2q};\la_2) F_{i+} 
      +c\frs( t_{1i}-t_{2q};\la_2) F_{i-} 
      +c\frs(-t_{2q} -t_{1i};\la_1) F_{+q} 
      -c\frs(t_{2q} -t_{1i};\la_1) F_{-q}  ,\\
    \pa_{2q} (F_{\pm q} )
    &= 
      \Big( \tpi (\mp c_{10} + c_{20}) - \sum_{j\neq q} c_{2j}\rho(t_{2j}-t_{2q}) 
      +\sum_{i} c_{1i}\rho (t_{2q} \pm t_{1i})      
      + 2c \rho(2t_{2q}) \Big) F_{\pm q} \\
    &\qquad -2c\frs(2t_{2q}; \pm \la_1) F_{\mp q}
      + \sum_{j\neq q} c_{2j} \frs(t_{2j}-t_{2q};\la_2 \mp \la_1) F_{\pm j} \\
    &\qquad -\sum_{i} \big( c_{1i}\frs (t_{2q} \pm t_{1i}; \pm \la_1) F_{iq} 
      +c_{1i} \frs(\mp t_{1i}-t_{2q};\la_2 \mp \la_1) F_{i\pm} \big) \\
    &\qquad \pm c \frs(-t_{2q};\la_2 \mp \la_1) F_{+-,1}
      \pm c \frs(\pm \frac{1}{2}-t_{2q};\la_2 \mp \la_1) F_{+-,2} \\
    &\qquad \pm b_{\pm} c \frs(\pm\frac{\tau}{2}-t_{2q};\la_2 \mp \la_1) F_{+-,3}
      \pm b_{\pm} c \frs(\pm \frac{1+\tau}{2}-t_{2q};\la_2 \mp \la_1) F_{+-,4} .
  \end{align*}
\end{Th}

In what follows, we prove this theorem. 
For a differential $2$-form $\vph (u;\la)$, we set 
\begin{align*}
  \na_{kp} (\vph)= \frac{\pa}{\pa t_{kp}}\vph -c_{kp}\frac{\pa}{\pa \la_k}\vph -c_{kp}\rho(u_k-t_{kp})\vph .
\end{align*}
Similarly to \cite[Section 4]{Mano-Watanabe}, we have
\begin{align*}
  \pa_{kp} \int_{\De} T(u) \vph(u;\la) = \int_{\De} T(u) \cdot \na_{kp}(\vph(u;\la)) .
\end{align*}
Therefore, it suffices to express each $\na_{kp} (\psi_{\ast})$ 
as a linear combination of $\{ \psi_{\ast} \}_{\ast \in \Psi}$ 
in the twisted cohomology group.

We mainly perform computations for $\na_{1p}$. 
These computations require many formulas for $\frs$ and $\rho$, 
which we list in Appendix \ref{sec:appendix-formulas}. 

\subsection{Computing $\na_{1p}(\psi_{ij})$, $\na_{1p}(\psi_{i \pm})$, $\na_{1p}(\psi_{\pm j})$ ($i\neq p$)}
For $i\neq p$, the differential forms 
$\psi_{ij}$, $\psi_{i \pm}$, $\psi_{\pm j}$
do not contain the variable $t_{1p}$. 
Thus, it suffices to compute their images under the operator 
$-c_{1p}\frac{\pa}{\pa \la_1} -c_{1p}\rho(u_1-t_{1p})$. 
Using (\ref{eq:s-diff-lambda}) and (\ref{eq:mano(38)}), 
we obtain 
\begin{align*}
  \na_{1p} (\psi_{ij})
  =c_{1p} \rho (t_{1p}-t_{1i}) \psi_{ij} -c_{1p}\frs(t_{1p}-t_{1i};\la_1) \psi_{pj}
\end{align*}
and 
\begin{align*}
  \na_{1p} (\psi_{i\pm})
  =c_{1p} \rho (t_{1p}-t_{1i}) \psi_{i\pm} -c_{1p} \frs(t_{1p}-t_{1i};\la_1 \mp \la_2) \psi_{p \pm} .
\end{align*}
The first equation can be regarded as obtained by multiplying the second equation of \cite[(37)]{Mano} 
by $\we \frs(u_2 -t_{2j};\la_2) du_2$. 
Similarly, we have 
\begin{align*}
  \na_{1p} (\psi_{\pm j})/(c_{1p} du_1 \we du_2) 
  &=\rho (t_{1p} \pm t_{2j}) \frs(u_1 \pm u_2;\la_1) \frs(u_2-t_{2j};\la_2 \mp \la_1 ) \\
  &\ \pm \big(
    \frs(\mp t_{1p} -t_{2j};\la_2 \mp \la_1 ) \frs(u_2 \pm t_{1p};\la_2 \mp \la_1 ) \frs(u_1 \pm u_2;\la_1) \\
  &\qquad -\frs(u_1 -t_{1p};\la_1)\frs(u_2 \pm t_{1p};\pm \la_1)\frs(u_2-t_{2j};\la_2 \mp \la_1) \big) .
\end{align*}
By applying (\ref{eq:for-psi(pmj)}), we obtain
\begin{align*}
  \na_{1p} (\psi_{\pm j})
  =c_{1p} \rho (t_{1p} \pm t_{2j})\psi_{\pm j} 
  \mp c_{1p}\frs(t_{2j} \pm t_{1p};\pm \la_1)\psi_{pj}
  \mp c_{1p}\frs(\mp t_{1p}-t_{2j};\la_2 \mp \la_1)\psi_{p \pm} .
\end{align*}

\subsection{Computing $\na_{1p}(\psi_{+-,m})$}
By the definition of $\psi_{+-,m}$, it suffices to consider $\psi'_{+-,m}$ 
which is introduced in Section \ref{subsec:differential-forms}. 
By utilizing (\ref{eq:s-diff-lambda}), (\ref{eq:mano(38)}) and Lemma \ref{lem:rho(a+b)+rho(a-b)}, we have 
\begin{align*}
  &\na_{1p} \Big(\frs(u_1+u_2;\frac{\la_1 +\la_2 +2w_m}{2}) \frs(u_1-u_2;\frac{\la_1-\la_2+2w_m}{2})du_1 \we du_2 \Big) \\
  &=\frac{c_{1p}}{2}\Big(
    -\frs(u_1-w_m;\frac{\la_1 +\la_2+2w_m}{2})\frs(w_m+u_2;\frac{\la_1 +\la_2+2w_m}{2})
    \frs(u_1-u_2;\frac{\la_1-\la_2+2w_m}{2}) \\
  &\qquad 
    -\frs(u_1-w_m;\frac{\la_1-\la_2+2w_m}{2})\frs(w_m-u_2;\frac{\la_1-\la_2+2w_m}{2})
    \frs(u_1+u_2;\frac{\la_1 +\la_2+2w_m}{2}) \\
  &\qquad 
    +2\big( \rho (u_1-w_m)+\varpi_m -\rho(u_1-t_{1p}) \big)
    \frs(u_1+u_2;\frac{\la_1 +\la_2+2w_m}{2}) \frs(u_1-u_2;\frac{\la_1-\la_2+2w_m}{2}) 
    \Big) du_1 \we du_2 ,
\end{align*}
where we set 
\begin{align*}
  \varpi_{m}=
  \begin{cases}
    0 & (m=1,2) \\
    -\pii & (m=3,4)
  \end{cases}
\end{align*}
Using the formulas (\ref{eq:for-psi(+-1)})--(\ref{eq:for-psi(+-4)}), 
we obtain 
\begin{align*}
  &(\na_{1p}(\psi_{+-,1}),\na_{1p}(\psi_{+-,2}),\na_{1p}(\psi_{+-,3}),\na_{1p}(\psi_{+-,4}))\\
  &=(\na_{1p}(\psi'_{+-,1}),\na_{1p}(\psi'_{+-,2}),\na_{1p}(\psi'_{+-,3}),\na_{1p}(\psi'_{+-,4}))\cdot M^{-1} \\
  &=-c_{1p} \Big(
    -\big( 2\frs(2t_{1p};\frac{\la_1-\la_2}{2}), 2\frs(2t_{1p};\frac{\la_1-\la_2+1}{2}), \\
  &\qquad \qquad 
    2e^{-\tpi t_{1p}}\frs(2t_{1p};\frac{\la_1-\la_2+\tau}{2}), 2e^{-\tpi t_{1p}}\frs(2t;\frac{\la_1-\la_2+1+\tau}{2}) \big)
    \psi_{p+} \\
  &\qquad 
    +\big( 2\frs(2t_{1p};\frac{\la_1+\la_2}{2}), 2\frs(2t_{1p};\frac{\la_1+\la_2+1}{2}),\\
  &\qquad \qquad 
    2e^{-\tpi t_{1p}}\frs(2t_{1p};\frac{\la_1+\la_2+\tau}{2}), 2e^{-\tpi t_{1p}}\frs(2t;\frac{\la_1+\la_2+1+\tau}{2}) \big)
    \psi_{p-} \\
  &\qquad 
  -\rho (t_{1p}) \big( 1,1,1,1\big) \psi_{+-,1} 
  -\rho (t_{1p}-\frac{1}{2}) \big( 1, 1, -1, -1 \big) \psi_{+-,2} \\
  &\qquad 
    -\big( \rho(t_{1p}-\frac{\tau}{2} ) -\pii \big) 
    \big(\ell, -\ell, \ell, -\ell \big) \psi_{+-,3} 
    -\big( \rho(t_{1p}-\frac{1+\tau}{2}) -\pii \big)
    \big( \ell, -\ell, -\ell, \ell \big) \psi_{+-,4} 
  \Big) M^{-1} .
\end{align*}
As stated in \cite[Example 4.3.2]{LV}\footnote{
  We modify the definition of the differential forms $\tilde{\om}_i$ ($i=1,2,3,4$) from \cite[Example 4.3.2]{LV}
  as follows: 
  \begin{align*}
    \tilde{\om}_1 =\frac{1}{4} (\om_1 +\om_2 +\om_3 +\om_4 ),\ 
    \tilde{\om}_2 =\frac{1}{4} (\om_1 +\om_2 -\om_3 -\om_4 ),\ 
    \tilde{\om}_3 =\frac{\ga}{4} (\om_1 -\om_2 +\om_3 -\om_4 ),\ 
    \tilde{\om}_4 =\frac{\ga}{4} (\om_1 -\om_2 -\om_3 +\om_4 ) . 
  \end{align*}
}, we have
\begin{align*}
  &\big(2\frs(2t;\frac{\la}{2}), 2\frs(2t;\frac{\la+1}{2}),
  2e^{-\tpi t}\frs(2t;\frac{\la+\tau}{2}), 2e^{-\tpi t}\frs(2t;\frac{\la+1+\tau}{2}) \big)
  M^{-1} \\
  &=\big(\frs(t;\la), \frs(t-\frac{1}{2};\la), \frs(t-\frac{\tau}{2};\la), \frs(t-\frac{1+\tau}{2};\la) \big) .
\end{align*}
Therefore, we obtain 
\begin{align*}
  &(\na_{1p}(\psi_{+-,1}),\na_{1p}(\psi_{+-,2}),\na_{1p}(\psi_{+-,3}),\na_{1p}(\psi_{+-,4}))\\
  &=-c_{1p} \Big(
    -\big( \frs(t_{1p};\la_1-\la_2), \frs(t_{1p}-\frac{1}{2};\la_1-\la_2),
    \frs(t_{1p}-\frac{\tau}{2};\la_1-\la_2), \frs(t_{1p}-\frac{1+\tau}{2};\la_1-\la_2) \big)
    \psi_{p+} \\
  &\qquad 
    +\big( \frs(t_{1p};\la_1+\la_2) , \frs(t_{1p}-\frac{1}{2};\la_1+\la_2) ,
    \frs(t_{1p}-\frac{\tau}{2};\la_1+\la_2) ,\frs(t_{1p}-\frac{1+\tau}{2};\la_1+\la_2) \big)
    \psi_{p-} \\
  &\qquad 
  -\rho (t_{1p}) \big( 1,0,0,0\big) \psi_{+-,1} 
  -\rho (t_{1p}-\frac{1}{2}) \big( 0, 1, 0, 0 \big) \psi_{+-,2} \\
  &\qquad 
    -\big( \rho(t_{1p}-\frac{\tau}{2} ) -\pii  \big) 
    \big(0, 0, 1, 0 \big) \psi_{+-,3} 
    -\big( \rho(t_{1p}-\frac{1+\tau}{2}) -\pii \big)
    \big( 0, 0, 0, 1 \big) \psi_{+-,4} 
  \Big).
\end{align*}

\subsection{Computing $\na_{1p}(\psi_{pj})$}
We use the notation ``$\equiv$'', when both sides are cohomologous. 
By considering $\na (\frs(u_1 -t_{1p};\la_1) \frs(u_2-t_{2j};\la_2) du_2) \equiv 0$, we have
\begin{align*}
  &\big( (1-c_{1p})\rho(u_1-t_{1p}) -\rho (u_1-t_{1p}-\la) \big)
    \frs(u_1 -t_{1p};\la_1) \frs(u_2 -t_{2j};\la_2) du_1 \we du_2 \\
  &\equiv \Big(\tpi c_{10}
    +c\rho(u_1-u_2)+c\rho(u_1+u_2) 
    +\sum_{i\neq p} c_{1i}\rho (u_1 -t_{1i}) \Big)
    \frs(u_1 -t_{1p};\la_1) \frs(u_2 -t_{2j};\la_2) du_1 \we du_2.
\end{align*}
We thus have 
\begin{align*}
  \na_{1p} (\psi_{pj} )
  &=(\rho (u_1-t_{1p})-\rho (u_1-t_{1p}-\la_1))\frs(u_1 -t_{1p};\la_1) \frs(u_2 -t_{2j};\la_2) du_1 \we du_2  \\
  &\quad 
    +c_{1p}\left( \rho (u_1-t_{1p}-\la_1) +\rho (\la_1) -\rho(u_1-t_{1p}) \right)
    \frs(u_1 -t_{1p};\la_1) \frs(u_2 -t_{2j};\la_2) du_1 \we du_2, \\
  &\equiv \Big( c_{1p}\big( \rho (u_1-t_{1p}-\la_1) +\rho (\la_1) \big)
    +\tpi c_{10}
    +c\rho(u_1-u_2)+c\rho(u_1+u_2) \\
  & \qquad 
    +\sum_{i\neq p} c_{1i}\rho (u_1 -t_{1i})  \Big) 
    \frs(u_1 -t_{1p};\la_1) \frs(u_2 -t_{2j};\la_2) du_1 \we du_2 . 
\end{align*}
By applying (\ref{eq:mano(39)}) and (\ref{eq:for-psi(pj)}), we obtain 
\begin{align*}
  \na_{1p} (\psi_{pj} )
  &\equiv \big( \tpi c_{10} +\sum_{i\neq p} c_{1i}\rho (t_{1p}-t_{1i}) \big) \psi_{pj}
    +\sum_{i\neq p} c_{1i} \frs(t_{1i}-t_{1p};\la_1) \psi_{ij} \\
  &\qquad +c\Big( \frs(u_2-t_{1p};\la_1)\frs(u_1 -u_2;\la_1)
    +\frs(-u_2-t_{1p};\la_1)\frs(u_1 +u_2;\la_1) \\
  &\qquad \qquad 
    -\big(\rho(u_2-t_{1p}) +\rho(-u_2-t_{1p}) \big)\frs(u_1 -t_{1p};\la_1) \Big) 
     \frs(u_2 -t_{2j};\la_2) du_1 \we du_2 \\
  &= \Big( \tpi c_{10} +\sum_{i\neq p} c_{1i}\rho (t_{1p}-t_{1i}) 
    +c \big(\rho(t_{1p}-t_{2j}) +\rho(t_{1p}+t_{2j}) \big) \Big) \psi_{pj}
    +\sum_{i\neq p} c_{1i} \frs(t_{1i}-t_{1p};\la_1) \psi_{ij} \\
  &\qquad +c\frs(- t_{2j}-t_{1p};\la_1) \psi_{+j} 
    +c\frs( t_{2j}-t_{1p};\la_1) \psi_{-j} 
    +c\frs(-t_{1p} -t_{2j};\la_2) \psi_{p+} 
    -c\frs(t_{1p} -t_{2j};\la_2) \psi_{p-} .
\end{align*}

\subsection{Computing $\na_{1p}(\psi_{p\pm})$}
By considering $\na (\frs(u_1-t_{1p};\la_1 \mp \la_2) \frs(u_1 \pm u_2; \pm \la_2) d(u_1 \pm u_2)) \equiv 0$, 
we have 
\begin{align*}
  &\pm \big( (1-c_{1p})\rho(u_1-t_{1p}) -\rho (u_1-t_{1p}-(\la_1 \mp \la_2)) \big)
    \frs(u_1-t_{1p};\la_1 \mp \la_2)\frs(u_1 \pm u_2; \pm\la_2) du_1 \we du_2 \\
  &\equiv \Big( \tpi (\pm c_{10}-c_{20}) \pm 2c\rho(u_1 \mp u_2) \\
  &\qquad 
    \pm \sum_{i\neq p} c_{1i}\rho (u_1 -t_{1i})-\sum_{j} c_{2j}\rho (u_2-t_{2j}) \Big)
    \frs(u_1-t_{1p};\la_1 \mp \la_2)\frs(u_1 \pm u_2;\pm \la_2) du_1 \we du_2 .
\end{align*}
Applying this identity together with (\ref{eq:mano(39)}) and (\ref{eq:for-psi(ppm)}), we obtain 
\begin{align*}
  \na_{1p} (\psi_{p\pm} )
  &=\pm \Big( 
    \big( \rho (u_1-t_{1p})-\rho (u_1-t_{1p}-(\la_1 \mp \la_2)) \big) \frs(u_1 -t_{1p};\la_1 \mp \la_2) 
    \frs(u_1 \pm u_2; \pm \la_2) du_1 \we du_2  \\
  &\quad +c_{1p}\big( \rho (u_1-t_{1p}-(\la_1 \mp \la_2)) +\rho (\la_1 \mp \la_2) -\rho(u_1-t_{1p}) \big) \\
  &\qquad 
    \cdot \frs(u_1-t_{1p};\la_1 \mp \la_2)\frs(u_1 \pm u_2; \pm \la_2) du_1 \we du_2
    \Big) \\
  &\equiv \Big( 
    \pm c_{1p}\left( \rho (u_1-t_{1p}-(\la_1 \mp \la_2)) +\rho (\la_1 \mp \la_2) \right)\\
  &\quad +( \tpi (\pm c_{10}-c_{20}) \pm 2c\rho(u_1 \mp u_2) 
    \pm \sum_{i\neq p} c_{1i}\rho (u_1 -t_{1i})-\sum_{j} c_{2j}\rho (u_2-t_{2j})) 
    \Big) \\
  &\qquad \cdot \frs(u_1-t_{1p};\la_1 \mp \la_2)\frs(u_1 \pm u_2; \pm \la_2) du_1 \we du_2 \\
  &=\Big( \tpi (c_{10} \mp c_{20}) - \sum_{i\neq p} c_{1i}\rho(t_{1i}-t_{1p}) \Big) \psi_{p\pm} 
    + \sum_{i\neq p} c_{1i} \frs(t_{1i}-t_{1p};\la_1 \mp \la_2) \psi_{i\pm} \\
  &\qquad +\Big( \big( 
    -\sum_{j} c_{2j}\rho (u_2-t_{2j})      
    \mp 2c \rho(\pm u_2-t_{1p})
    \big)\frs(u_1-t_{1p};\la_1 \mp \la_2) \\
  &\qquad \qquad 
    \pm 2c \frs(u_1\mp u_2;\la_1 \mp \la_2)\frs(\pm u_2-t_{1p};\la_1 \mp \la_2)
    \Big) \frs(u_1 \pm u_2; \pm \la_2) du_1 \we du_2 \\
  &=\Big( \tpi (c_{10} \mp c_{20}) - \sum_{i\neq p} c_{1i}\rho(t_{1i}-t_{1p}) 
    +\sum_{j} c_{2j}\rho (t_{1p} \pm t_{2j})      
    + 2c \rho(2t_{1p}) \Big) \psi_{p\pm} \\
  &\qquad -2c\frs(2t_{1p}; \pm \la_2) \psi_{p\mp}
    + \sum_{i\neq p} c_{1i} \frs(t_{1i}-t_{1p};\la_1 \mp \la_2) \psi_{i\pm} \\
  &\qquad -\sum_{j} \big( c_{2j}\frs (t_{1p} \pm t_{2j}; \pm \la_2 ) \psi_{pj} 
    +c_{2j} \frs(\mp t_{2j}-t_{1p};\la_1 \mp \la_2) \psi_{\pm j} \big) \\
  &\qquad \mp c \frs(-t_{1p};\la_1 \mp \la_2) \psi_{+-,1}
    \mp c \frs(\pm \frac{1}{2}-t_{1p};\la_1 \mp \la_2) \psi_{+-,2} \\
  &\qquad \mp a_{\pm} c \frs(\pm\frac{\tau}{2}-t_{1p};\la_1 \mp \la_2) \psi_{+-,3}
    \mp a_{\pm} c \frs(\pm \frac{1+\tau}{2}-t_{1p};\la_1 \mp \la_2) \psi_{+-,4} .
\end{align*}

\subsection{Computing $\na_{2q}$}
We compute $\na_{2q}$ by changing $u_1 \leftrightarrow u_2$. 
Recall that $\vth_1 (-u)=-\vth_1(u)$. 
We formally replace the variables in the following rule:
\begin{align*}
  (u_1,u_2,\la_1,\la_2,t_{1i},t_{2j}) \to (u_2,u_1,\la_2,\la_1,t_{2j},t_{1i}).
\end{align*}
For a function or differential form $\vph$, 
let $\vph^{\star}$ denote that obtained by this replacement. 
By straightforward calculation, we have 
\begin{align*}
  &(\ell)^{\star} =\ell, \qquad 
    (a_{\pm})^{\star}=b_{\pm},\\ 
  &(\psi_{ij})^{\star}=-\psi_{ij} ,\qquad 
    (\psi_{i\pm})^{\star} =-\psi_{\pm j} ,\qquad 
    (\psi_{\pm j})^{\star}=-\psi_{i\pm}, \qquad 
    (\psi_{+-,m})^{\star}=\psi_{+-,m}. 
\end{align*}
Using these relations, we can calculate $\na_{2q}(\psi_{\ast})$ from 
the results for $\na_{1p}(\psi_{\ast})$ which we have calculated above.

\appendix
\section{Some formulas}\label{sec:appendix-formulas}
In this appendix, we list the formulas used in Section \ref{sec:diff-eq}. 

\subsection{Basic formulas}
By a straightforward computation, we have
\begin{align}
  \label{eq:rho-period}
  &\rho (u+1)=\rho(u),\qquad 
  \rho (u+\tau)=\rho(u)-\tpi ,\\
  \label{eq:s(-u)=-s(u)}
  &\frs(-u;\la)
  =-\frs(u;-\la). 
\end{align}
By considering the logarithmic derivatives, we obtain 
\begin{align}
  \label{eq:s-diff-u}
  \frac{\pa}{\pa u} \frs(u;\la)
  &=(\rho (u-\la) -\rho (u) )\frs(u;\la) ,\\
  \label{eq:s-diff-lambda}
  \frac{\pa}{\pa \la} \frs(u;\la)
  &=-(\rho (u-\la) +\rho (\la) )\frs(u;\la).
\end{align}
Let $t_1,\dots ,t_n \in \C$ be points that represent distinct points of $\C/(\Z +\Z \tau)$, 
and $c_1 ,\dots ,c_n \in \C$ satisfy $c_1+\dots +c_n =0$.  
For distinct $j$, $k$ and $l$, 
we have 
\begin{align}
  \label{eq:mano(38)}
  &\frs(w-t_k;\la) (\rho(w-t_j)+\rho(t_j-t_k) -\rho(w-t_k-\la) -\rho(\la))
  =\frs(w-t_j;\la)\frs(t_j-t_k;\la) , \\
  \nonumber
  &\frs(w-t_j;\la) \Big( \sum_{k\neq j} c_k(\rho(w-t_k)+\rho(t_k-t_j)) +c_j(\rho(w-t_j-\la) +\rho(\la)) \Big)\\
  \label{eq:mano(39)}
  & =\sum_{k\neq j} c_k \frs(t_k-t_j;\la) \frs(w-t_k;\la)  
\end{align}
(see \cite[(38),(39)]{Mano}). 

\begin{Lem}\label{lem:rho(a+b)+rho(a-b)}
  \begin{enumerate}[(i)]
  \item $\rho(\frac{1}{2})=0$, $\rho(\frac{\tau}{2})=\rho(\frac{1+\tau}{2})=-\pii$. 
  \item For $m\in \{ 2,3,4\}$ and $t\in \C$, we have $\DS \rho(w_m +t)+\rho(w_m -t)=2\rho(w_m)(=2\varpi_m)$.
  \end{enumerate} 
\end{Lem}
\begin{proof}
  \begin{enumerate}[(i)]
  \item Since $\rho(u)$ is an odd function, 
    we can prove the formulas by
    substituting $u=-\frac{1}{2},-\frac{\tau}{2},-\frac{1+\tau}{2}$ into (\ref{eq:rho-period}). 
  \item As a function of $t$, the left-hand side is an elliptic function. 
    We can easily show that this function is holomorphic, and hence it is constant.  
    \qedhere
  \end{enumerate}
\end{proof}

\subsection{Formulas to compute $\na_{1p}$}
To prove several formulas, 
we use 
\begin{align}
  \label{eq:rel-s-FRV}
  \frs(t-u;\la_1+\la_2)\frs(s-t;\la_2) - \frs(s-u;\la_2)\frs(t-u;\la_1) + \frs(t-s;\la_1)\frs(s-u;\la_1+\la_2) =0 
\end{align}
which is given in \cite[Theorem 5.3]{FRV}. 
The same method is also used in our proof. 

\begin{Lem}\label{lem:for-psi(pmj)}
  \begin{align}
    \nonumber
    &\frs(\mp t_{1p} -t_{2j};\la_2 \mp \la_1 ) \frs(u_2 \pm t_{1p};\la_2 \mp \la_1 ) \frs(u_1 \pm u_2;\la_1) \\
    \nonumber
    &\qquad -\frs(u_1 -t_{1p};\la_1)\frs(u_2 \pm t_{1p};\pm \la_1)\frs(u_2-t_{2j};\la_2 \mp \la_1) \\
    \nonumber
    &=-\frs(t_{2j} \pm t_{1p};\pm \la_1)\frs(u_1 -t_{1p};\la_1) \frs(u_2-t_{2j};\la_2) \\
    \label{eq:for-psi(pmj)}
    &\qquad \mp \frs(\mp t_{1p}-t_{2j};\la_2 \mp \la_1)\frs(u_1-t_{1p};\la_1 \mp \la_2) \frs(u_1 \pm u_2; \pm \la_2) .
  \end{align}  
\end{Lem}
\begin{proof}
  We give an outline of the proof. 
  Let $f$ denote the function defined by $(\text{LHS})-(\text{RHS})$. 
  We regard $f$ as a function of $u_1$. 
  Using (\ref{eq:rel-s-FRV}), we can show that 
  $\Res_{u_1=t_{1p}}f =\Res_{u_1=u_2}f =0$, and hence $f$ is an entire function. 
  Since $f(u_1 +1)=f(u_1)$ and $f(u_1 +\tau)=e^{\tpi \la_1}f(u_1)$, 
  the same argument as the proof of \cite[Theorem 5.3]{FRV} implies $f=0$. 
\end{proof}

\begin{Lem}\label{lem:for-psi(+-m)}
  For $\ast \in \Psi$, let $g_{\ast}(u_1,u_2)$ be the function defined by 
  $\psi_{\ast}/(du_1 \we du_2)$. 
  The functions 
  \begin{align*}
    G_m&=-\frs(u_1-w_m;\frac{\la_1 +\la_2+2w_m}{2})\frs(w_m+u_2;\frac{\la_1 +\la_2+2w_m}{2})
    \frs(u_1-u_2;\frac{\la_1-\la_2+2w_m}{2}) \\
    &\quad 
      -\frs(u_1-w_m;\frac{\la_1-\la_2+2w_m}{2})\frs(w_m-u_2;\frac{\la_1-\la_2+2w_m}{2})
      \frs(u_1+u_2;\frac{\la_1 +\la_2+2w_m}{2}) \\
    &\quad 
      +2\big( \rho (u_1-w_m)+\varpi_m -\rho(u_1-t_{1p}) \big)
      \frs(u_1+u_2;\frac{\la_1 +\la_2+2w_m}{2}) \frs(u_1-u_2;\frac{\la_1-\la_2+2w_m}{2}) 
  \end{align*}
  are expressed as follows: 
  \begin{align}
    \nonumber
    G_1
    &=-2\frs(2t_{1p};\frac{\la_1-\la_2}{2})g_{p+}
      +2\frs(2t_{1p};\frac{\la_1 +\la_2}{2})g_{p-}
      -\rho(t_{1p})g_{+-,1} -\rho(t_{1p}-\frac{1}{2}) g_{+-,2} \\
    \label{eq:for-psi(+-1)}
    &\quad -\ell \Big( \rho(t_{1p}-\frac{\tau}{2}) -\pii \Big)g_{+-,3}
      -\ell \Big( \rho(t_{1p}-\frac{1+\tau}{2}) -\pii \Big)g_{+-,4} ,\\
    \nonumber
    G_2
    &=-2\frs(2t_{1p};\frac{\la_1-\la_2+1}{2})g_{p+}
      +2\frs(2t_{1p};\frac{\la_1 +\la_2+1}{2})g_{p-} 
      -\rho(t_{1p}) g_{+-,1} -\rho(t_{1p}-\frac{1}{2}) g_{+-,2} \\
    \label{eq:for-psi(+-2)}
    &\qquad
      +\ell \Big( \rho(t_{1p}-\frac{\tau}{2})-\pii \Big)g_{+-,3}
      +\ell \Big( \rho(t_{1p}-\frac{1+\tau}{2})-\pii \Big)g_{+-,4} ,
  \end{align}
  \begin{align}
    \nonumber
    &e^{-\tpi u_1}G_3 \\
    \nonumber
    &=-2 e^{-\tpi t_{1p}} \frs(2t_{1p};\frac{\la_1-\la_2+\tau}{2})g_{p+}
      +2 e^{-\tpi t_{1p}} \frs(2t_{1p};\frac{\la_1 +\la_2+\tau}{2})g_{p-} \\
    \label{eq:for-psi(+-3)}
    &\quad -\rho(t_{1p}) g_{+-,1}  +\rho(t_{1p}-\frac{1}{2})g_{+-,2}
      -\ell \Big( \rho(t_{1p}-\frac{\tau}{2}) -\pii \Big)g_{+-,3} 
      +\ell \Big( \rho(t_{1p}-\frac{1+\tau}{2}) -\pii \Big) g_{+-,4},\\
    \nonumber
    &e^{-\tpi u_1}G_4 \\
    \nonumber
    &=-2 e^{-\tpi t_{1p}} \frs(2t_{1p};\frac{\la_1-\la_2+1+\tau}{2})g_{p+}
      +2 e^{-\tpi t_{1p}} \frs(2t_{1p};\frac{\la_1 +\la_2+1+\tau}{2})g_{p-} \\
    \label{eq:for-psi(+-4)}
    &\quad -\rho(t_{1p}) g_{+-,1}  +\rho(t_{1p}-\frac{1}{2}) g_{+-,2}
      +\ell \Big( \rho(t_{1p}-\frac{\tau}{2}) -\pii \Big) g_{+-,3} 
      -\ell \Big( \rho(t_{1p}-\frac{1+\tau}{2}) -\pii \Big) g_{+-,4}.
  \end{align}
\end{Lem}
\begin{proof}
  We show only (\ref{eq:for-psi(+-1)}). 
  Once we have proved (\ref{eq:for-psi(+-1)}), the other three formulas 
  can be obtained by replacing 
  $(u_1,u_2,\la_1 ,t_{1p})$ with $(u_1-w_m,u_2-w_m,\la_1+2w_m ,t_{1p}-w_m)$. 

  We can prove (\ref{eq:for-psi(+-1)}) similarly to the proof of Lemma \ref{lem:for-psi(pmj)}: 
  we set $f(u_1)=(\text{LHS})-(\text{RHS})$ and show that 
  $\Res_{u_1=t_{1p}}f=\Res_{u_1 =\pm u_2}f=0$. 
  Note that $f$, regarded as a function of $u_1$, does not have a pole along $(u_1=0)$. 
  Indeed, we can verify this by considering the residue. 

  We give some remarks for computing $\Res_{u_1 =\pm u_2}f$. 
  By definition, the residue $\Res_{u_1-u_2=0}(\Res_{u_1+u_2=0} \psi_{+-,m})$ equals $1$ when 
  we take $(u_1,u_2)=(w_m,w_m)$ as a representative of $P_m \in E^2$. 
  However, if we consider $\Res_{u_2=w_m}(\Res_{u_1=-u_2} g_{+-,m}(u_1,u_2))$, 
  the corresponding representative is $(u_1,u_2)=(-w_m,w_m)$. 
  Thus, the residue $\Res_{u_2=w_m}(\Res_{u_1=-u_2} g_{+-,m}(u_1,u_2))$ may not equal $1$. 
  We list the residues:
  \begin{align*}
    \begin{array}{r|c|c|c|c}
      m&1&2&3&4 \\ \hline
      \Res_{u_2=w_m}(\Res_{u_1=-u_2} g_{+-,m}) &1&1&e^{\tpi (-\la_1)}&e^{\tpi (-\la_1)} \\
      \Res_{u_2=w_m}(\Res_{u_1=u_2} g_{+-,m}) &-1&-1&-1&-1
    \end{array}
  \end{align*}
  When we regard $\Res_{u_1 =\pm u_2}f$ as a function of $u_2$, 
  it seems that $\Res_{u_1 =\pm u_2}f$ has a pole of order $2$ at $u_2=0$. 
  However, we can see that its singular part vanishes. 
\end{proof}

The following two lemmas can be proved similarly. 
\begin{Lem}
  \begin{align}
    \nonumber 
    &\Big(\frs(u_2-t_{1p};\la_1)\frs(u_1 -u_2;\la_1)
      +\frs(-u_2-t_{1p};\la_1)\frs(u_1 +u_2;\la_1) \\
    &\nonumber \qquad 
      -\big(\rho(u_2-t_{1p}) +\rho(-u_2-t_{1p}) \big)\frs(u_1 -t_{1p};\la_1) \Big)  
      \cdot \frs(u_2 -t_{2j};\la_2) \\
    &\nonumber 
      =\big( \rho(t_{1p}-t_{2j}) +\rho(t_{1p}+t_{2j}) \big) \frs(u_1 -t_{1p};\la_1) \frs(u_2-t_{2j};\la_2) \\
    &\nonumber \qquad 
      +\frs(- t_{2j}-t_{1p};\la_1) \frs(u_1 + u_2;\la_1) \frs(u_2-t_{2j};\la_2 - \la_1)\\
    &\nonumber \qquad  
      +\frs( t_{2j}-t_{1p};\la_1) \frs(u_1 - u_2;\la_1) \frs(u_2-t_{2j};\la_2 + \la_1) \\
    &\nonumber \qquad 
      +\frs(-t_{1p} -t_{2j};\la_2) \frs(u_1-t_{1p};\la_1 - \la_2) \frs(u_1 + u_2; \la_2)\\
    &\label{eq:for-psi(pj)} \qquad  
      +\frs(t_{1p} -t_{2j};\la_2) \frs(u_1-t_{1p};\la_1 + \la_2) \frs(u_1 - u_2; - \la_2) .
  \end{align}
\end{Lem}

\begin{Lem}
  Let $g_{\ast}$ ($\ast \in \Psi$) be the function defined in Lemma \ref{lem:for-psi(+-m)}. 
  Then, we have 
  \begin{align}
    &\nonumber 
      \Big( \big( 
      -\sum_{j} c_{2j}\rho (u_2-t_{2j})      
      \mp 2c \rho(\pm u_2-t_{1p})
      \big)\frs(u_1-t_{1p};\la_1 \mp \la_2)  \\
    &\nonumber \qquad 
      \pm 2c \frs(u_1\mp u_2;\la_1 \mp \la_2)\frs(\pm u_2-t_{1p};\la_1 \mp \la_2)
      \Big) \frs(u_1 \pm u_2; \pm \la_2)  \\
    &\nonumber 
      =-\sum_{j} \big( c_{2j}\frs (t_{1p} \pm t_{2j}; \pm \la_2) g_{pj} 
      +c_{2j} \frs(\mp t_{2j}-t_{1p};\la_1 \mp \la_2) g_{\pm j} \big) \\
    &\nonumber \qquad
      + \Big( \sum_{j} c_{2j}\rho (t_{1p} \pm t_{2j})      
      + 2c \rho(2t_{1p}) \Big) g_{p\pm}
      -2c\frs(2t_{1p}; \pm \la_2) g_{p\mp} \\
    &\nonumber \qquad
      \mp c \frs(-t_{1p};\la_1 \mp \la_2) g_{+-,1}
      \mp c \frs(\pm \frac{1}{2}-t_{1p};\la_1 \mp \la_2) g_{+-,2} \\
    &\label{eq:for-psi(ppm)} \qquad
      \mp a_{\pm} c \frs(\pm\frac{\tau}{2}-t_{1p};\la_1 \mp \la_2) g_{+-,3}
      \mp a_{\pm} c \frs(\pm \frac{1+\tau}{2}-t_{1p};\la_1 \mp \la_2) g_{+-,4} .
  \end{align}
\end{Lem}

\begin{Ack}
  The author is grateful to Professor Saiei-Jaeyeong Matsubara-Heo for helpful advice. 
  This work was supported by JSPS KAKENHI Grant Number JP24K06680.
\end{Ack}


\begin{thebibliography}{99}
\bibitem{AK}
  K. Aomoto and M. Kita, 
  ``Theory of Hypergeometric Functions'', 
  translated by K. Iohara, 
  Springer Monographs in Mathematics,
  Springer-Verlag, Tokyo, 2011. 

\bibitem{Cho}
  K. Cho,
  A generalization of Kita and Noumi's vanishing theorems of cohomology groups of local system, 
  \textit{Nagoya Math. J.} \textbf{147} (1997), 63--69.


\bibitem{Deligne}
  P. Deligne, 
  ``\'Equations diff\'erentielles \`a points singuliers r\'eguliers'',
  Lecture Notes in Math., Vol. 163, 
  Springer-Verlag, Berlin-New York, 1970.

\bibitem{FRV}
  G. Felder, R. Rim\'{a}nyi and A. Varchenko, 
  Poincar\'{e}-Birkhoff-Witt expansions of the canonical elliptic differential form, 
  Quantum groups, 191--208.
  Contemp. Math., 433, 
  Israel Math. Conf. Proc.
  American Mathematical Society, Providence, RI, 2007. 


\bibitem{G-RWintegral-intersection}
  Y. Goto, 
  Intersection numbers of twisted homology and cohomology groups associated to the Riemann-Wirtinger integral, 
  \textit{Internat. J. Math.} \textbf{34} (2023), no. 3, 32 pp.


\bibitem{LV}
  A. Levin, A. Varchenko, 
  Cohomology of the complement to an elliptic arrangement, 
  Configuration spaces, CRM Series 14, Ed. Norm., Pisa (2012), pp. 373--388

\bibitem{Mano}
  T. Mano, 
  The Riemann-Wirtinger Integral and Monodromy-Preserving Deformation on Elliptic Curves, 
  \textit{Int. Math. Res. Not. IMRN}, \textbf{2008}, Art. ID rnn110, 19 pp. 

\bibitem{M-k-form}
  K. Matsumoto, 
  Intersection numbers for logarithmic $k$-forms, 
  \emph{Osaka J. Math.}, \textbf{35} (1998), 873--893. 

\bibitem{Mano-Watanabe}
  T. Mano and H. Watanabe, 
  Twisted cohomology and homology groups associated to the Riemann-Wirtinger integral,
  \textit{Proc. Amer. Math. Soc.}, \textbf{140} (2012), no. 11, 3867--3881.  

\bibitem{Matsubara-localization}
  S.-J. Matsubara-Heo, 
  Localization formulas of cohomology intersection numbers, 
  \textit{J. Math. Soc. Japan}, \textbf{75} (2023), no. 3, 909--940.




\end{thebibliography}
\end{document}